\documentclass{amsart}
\usepackage[all]{xy}
\usepackage{verbatim}
\usepackage{color}
\usepackage{amsthm}
\usepackage{amssymb}
\usepackage[colorlinks=true]{hyperref}
\usepackage{graphicx}
\usepackage{mathtools}

\numberwithin{equation}{section}

\newtheorem{theorem}[equation]{Theorem}
\newtheorem*{theorem*}{Theorem}

\newtheorem*{conjecture*}{Mamma Conjecture}
\newtheorem*{conjecture1*}{Mamma Conjecture (revisited)}

\newtheorem{corollary}[equation]{Corollary}
\newtheorem*{corollary*}{Corollary}

\theoremstyle{remark}

\theoremstyle{remark}

\newcommand{\cC}{{\mathcal C}}
\newcommand{\cD}{{\mathcal D}}

\newcommand{\bbA}{\mathbb{A}}

\newcommand{\bbZ}{\mathbb{Z}}

\DeclareMathOperator{\Id}{Id}

\newcommand{\bbK}{I\mspace{-6.mu}K}

\newcommand{\dg}{\mathrm{dg}}

\newcommand{\too}{\longrightarrow}

\let\oldmarginpar\marginpar
\def\marginpar#1{\oldmarginpar{\tiny #1}}

\def\multiset#1#2{\ensuremath{\left(\kern-.3em\left(\genfrac{}{}{0pt}{}{#1}{#2}\right)\kern-.3em\right)}}

\begin{document}

\title[$K$-theory with coefficients of cyclic quotient singularities]{Algebraic $K$-theory with coefficients of \\cyclic quotient singularities}
\author{Gon{\c c}alo~Tabuada}

\address{Gon{\c c}alo Tabuada, Department of Mathematics, MIT, Cambridge, MA 02139, USA}
\email{tabuada@math.mit.edu}
\urladdr{http://math.mit.edu/~tabuada}
\thanks{The author was partially supported by a NSF CAREER Award}

\date{\today}


\abstract{In this short note, by combining the work of Amiot-Iyama-Reiten and Thanhoffer de V\"olcsey-Van den Bergh on Cohen-Macaulay modules with the previous work of the author on orbit categories, we  compute the (nonconnective) algebraic $K$-theory with coefficients of cyclic quotient singularities.}
}

\maketitle
\vskip-\baselineskip
\vskip-\baselineskip

\section{Introduction and statement of results}\label{sec:introduction}
Let $k$ be an algebraically closed field of characteristic zero. Given an integer $d\geq 2$, consider the associated polynomial ring $S:=k[t_1, \ldots, t_d]$. Let $G$ be a cyclic subgroup of $\mathrm{SL}(d,k)$ generated by $\mathrm{diag}(\zeta^{a_1}, \ldots, \zeta^{a_d})$, where $\zeta$ is a primitive $n^{\mathrm{th}}$ root of unit and $a_1, \ldots, a_d$ are integers satisfying the following conditions: we have $0 < a_j < n$ and $\mathrm{gcd}(a_j,n)=1$ for every $1 \leq j \leq d$; we have $a_1 + \cdots + a_d=n$. The group $G$ acts naturally on $S$ and the invariant ring $R:=S^G$ is a Gorenstein isolated singularity of Krull dimension $d$. For example, when $d=2$, the ring $R$ identifies with the Kleinian singularity $k[u,v,w]/(u^n+vw)$ of type $A_{n-1}$. 

The affine $k$-scheme $X:=\mathrm{Spec}(R)$ is singular. Following Orlov \cite{Orlov, Orlov1}, we can then consider the associated dg category of singularities $\cD^{\mathrm{sing}}_\dg(X)$; also known as matrix factorizations or maximal Cohen-Macaulay modules. Roughly speaking, this dg category encodes all the crucial information concerning the isolated~singularity~of~$X$. 

Let us denote by $(Q,\rho)$ the quiver with relations defined by the following steps:
\begin{itemize}
\item[(s1)] consider the quiver with vertices $\bbZ/n\bbZ$ and with arrows $x^i_j\colon i \to i + a_j$, where $i \in \bbZ/n\bbZ$ and $1 \leq j \leq d$. The relations $\rho$ are given by $x^{i+a_j}_{j'}x^i_j = x^{i+a_{j'}}_j x_j^i$ for every $i \in \bbZ/n\bbZ$ and $1\leq j,j' \leq d$.
\item[(s2)] remove from (s1) all arrows $x^i_j\colon i \to i'$ with $i > i'$;
\item[(s3)] remove from (s2) the vertex $0$.
\end{itemize}
Consider the matrix $(n-1)\times (n-1)$ matrix $C$ such that $C_{ij}$ equals the number of arrows in $Q$ from $j$ to $i$ (counted modulo the relations). Let us write $M$ for the matrix $(-1)^{d-1}C(C^{-1})^T-\Id$ and $\mathrm{M}\colon \oplus_{r=1}^{n-1} \bbZ/l^\nu \to \oplus_{r=1}^{n-1} \bbZ/l^\nu$ for the associated (matrix) homomorphism, where $l^\nu$ is a (fixed) prime power. 
\begin{theorem}\label{thm:main}
We have the following computation:
$$
\bbK_i(\cD^{\mathrm{sing}}_\dg(X); \bbZ/l^\nu) \simeq \left\{ \begin{array}{lll}
\mathrm{cokernel}\,\,\mathrm{of}\,\, \mathrm{M}& \mathrm{if} & i\geq 0\,\,\mathrm{even} \\
\mathrm{kernel}\,\,\mathrm{of}\,\,\mathrm{M} & \mathrm{if} & i\geq 0 \,\,\mathrm{odd} \\
0 & \mathrm{if} & i <0\,.
\end{array} \right.
$$
\end{theorem}
Thanks to Theorem \ref{thm:main}, the computation of the (nonconnective) algebraic $K$-theory with coefficients of the cyclic quotient singularities reduces to the computation of (co)kernels of explicit matrix homomorphisms! To the best of the author's knowledge, these computations are new in the literature. In the particular case of Kleinian singularities of type $A_n$ they were originally established in \cite[\S3]{Klein}
\begin{corollary}\label{cor:main}
(i) If there exists a prime power $l^\nu$ and an even (resp. odd) integer $j \geq 0$ such that $\bbK_j(\cD^{\mathrm{sing}}_\dg(X); \bbZ/l^\nu)\neq 0$, then for every even (resp. odd) integer $i \geq 0$ at least one of the groups $\bbK_i(\cD^{\mathrm{sing}}_\dg(X)), \bbK_{i-1}(\cD_\dg^{\mathrm{sing}}(X))$ is non-zero.

(ii) If there exists a prime power $l^\nu$ such that $\bbK_i(\cD^{\mathrm{sing}}_\dg(X);\bbZ/l^\nu)=0$ for every $i \geq 0$, then the groups $\bbK_i(\cD^{\mathrm{sing}}_\dg(X)), i \geq 0$, are uniquely $l^\nu$-divisible.
\end{corollary}
\begin{proof}
Combine the universal coefficients sequence (see \cite[\S5]{Klein})
\begin{equation*}
0 \too \bbK_i(\cD^{\mathrm{sing}}_\dg(X))\otimes_\bbZ \bbZ/l^\nu \too  \bbK_i(\cD^{\mathrm{sing}}_\dg(X); \bbZ/l^\nu) \too {}_{l^\nu}\bbK_{i-1}(\cD^{\mathrm{sing}}_\dg(X)) \too 0
\end{equation*}
with the computation of Theorem \ref{thm:main}.
\end{proof}
\section{Examples}
\subsection*{A low dimensional example}
When $d=3$, $n=5$, $a_1=1$, and $a_2=a_3=2$, the above three steps (s1)-(s3) lead to the following quiver
$$ 
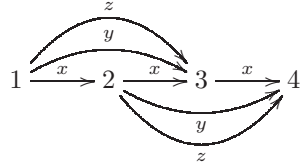
\xymatrix{
1 \ar[r]^x \ar@/^1pc/[rr]^-y \ar@/^2pc/[rr]^-z& 2 \ar@/_1pc/[rr]_-y \ar[r]^-x \ar@/_2pc/[rr]_-z & 3 \ar[r]^-x & 4
}
$$
with relations $xy=yx$, $yz=zy$, and $zx=xz$. Consequently, we obtain the matrix
$$
M=\begin{pmatrix}
0 & -1 & -3 &-3 \\
1 & -1 & -4 & -6 \\
3 & -2 & -10 & -13 \\
3 & 0 & -11 & -19
\end{pmatrix}\,.
$$
Since $\mathrm{det}(M)=26$, we have $\bbK_i(\cD^{\mathrm{sing}}_\dg(X); \bbZ/l^\nu)= 0$ whenever $l \neq 2, 13$. In the remaining two cases, a computation shows that $\bbK_i(\cD^{\mathrm{sing}}_\dg(X); \bbZ/l^\nu)\simeq \bbZ/l$ for every $i \geq 0$. Thanks to Corollary \ref{cor:main}, this implies that for every $i \geq 0$ at least one of the groups $\bbK_i(\cD^{\mathrm{sing}}_\dg(X)), \bbK_{i-1}(\cD^{\mathrm{sing}}_\dg(X))$ is non-trivial. Moreover, the groups $\bbK_i(\cD^{\mathrm{sing}}_\dg(X)), i \geq 0$, are uniquely $l$-divisible for every prime $l\neq 2, 13$.
\subsection*{A family of examples}
When $n=d\geq 3$ and $a_1=\cdots = a_d=1$, the above three steps (s1)-(s3) lead to the following quiver
$$
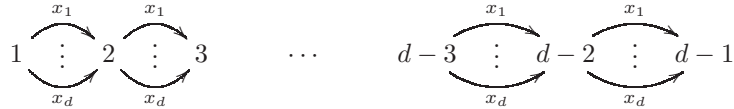
\xymatrix{
1 \ar@/^1pc/[r]^-{x_1} \ar@/_1pc/[r]_-{x_d} \ar@{}[r]|{\underset{}{\vdots}}& 2 \ar@/^1pc/[r]^-{x_1} \ar@{}[r]|{\underset{}{\vdots}} \ar@/_1pc/[r]_-{x_d}  &  3 & \cdots & d-3 \ar@/^1pc/[r]^-{x_1} \ar@{}[r]|{\underset{}{\vdots}} \ar@/_1pc/[r]_-{x_d}  &  d-2 \ar@/^1pc/[r]^-{x_1} \ar@{}[r]|{\underset{}{\vdots}} \ar@/_1pc/[r]_-{x_d}  & d-1
}
$$
with relations $x_j x_i = x_i x_j$. In the case where $d$ is odd, we obtain the matrix
$$
M_{ij}=\left\{ \begin{array}{lll}
-\sum_{r=0}^{i-1}\multiset{d}{r}\multiset{d}{(j-i)+r} & \mathrm{if}\,\,i< j\\
-\sum_{r=1}^{i-1} \multiset{d}{r}^{\!\!2}& \mathrm{if}\,\,i=j \\
- \sum_{r=1}^{j-1} \multiset{d}{(i-j)+r}\multiset{d}{r} + \multiset{d}{i-j} & \mathrm{if}\,\,i > j\,,
\end{array} \right.
$$
where $\multiset{}{}$ stands for the multicombination\footnote{Also known as the {\em multisubset} symbol.} symbol. Similarly, in the case where $d$ is even, we obtain the matrix
$$
M_{ij}=\left\{ \begin{array}{lll}
\sum_{r=0}^{i-1}\multiset{d}{r}\multiset{d}{(j-i)+r} & \mathrm{if}\,\,i< j\\
-2 + \sum^{i-1}_{r=1} \multiset{d}{r}^{\!\!2}& \mathrm{if}\,\,i=j \\
\sum_{r=1}^{j-1} \multiset{d}{(i-j)+r}\multiset{d}{r} - \multiset{d}{i-j}& \mathrm{if}\,\,i > j\,.
\end{array} \right.
$$
Whenever $d$ is a prime number, all the multicombinations
\begin{eqnarray*}
\multiset{d}{r}=\binom{d+r-1}{r} = \frac{(d+r-1) \cdots d (d-1)!}{r!(d-1)!} && 0 \leq r \leq d-2
\end{eqnarray*}
are multiples of $d$. This implies that the homomorphism $\mathrm{M}\colon \oplus^{d-1}_{r=1} \bbZ/d \to \oplus^{d-1}_{r=1} \bbZ/d$ is zero, and consequently that $\bbK_i(\cD^{\mathrm{sing}}_\dg(X);\bbZ/d)\simeq \oplus_{r=1}^{d-1} \bbZ/d$ for every $i \geq 0$. These isomorphisms are a far reaching generalization of the particular case $d=3$ originally established in \cite[Prop.~3.4]{Klein}. Thanks to Corollary \ref{cor:main}(i), we hence conclude that for every $i \geq 0$ at least one of the groups $\bbK_i(\cD^{\mathrm{sing}}_\dg(X)), \bbK_{i-1}(\cD^{\mathrm{sing}}_\dg(X))$ is non-trivial. 
\section{Proof of Theorem \ref{thm:main}}
Let $A$ be a finite dimensional $k$-algebra of finite global dimension. We write $\cD^b(A)$ for the bounded derived category of (right) $A$-modules and $\cD^b_\dg(A)$ for the canonical dg enhancement of $\cD^b(A)$. Consider the following dg functors
\begin{eqnarray*}
\tau^{-1}\Sigma^d\colon \cD^b_\dg(A) \too \cD^b_\dg(A) && d \geq 0\,,
\end{eqnarray*}
where $\tau$ stands for the Auslander-Reiten translaction. Following Keller \cite[\S7.2]{Orbit-Keller}, we can consider the associated dg orbit category $\cC_A^{(d)}:=\cD^b_\dg(A)/(\tau^{-1}\Sigma^d)^\bbZ$. Similarly to \cite[Thm.~2.5]{Klein} (consult \cite[\S2]{Orbit}), we have a distinguished triangle of spectra
$$\bigoplus^v_{r=1} \bbK(k;\bbZ/l^\nu) \stackrel{(-1)^d \Phi_A - \Id}{\too} \bigoplus^v_{r=1} \bbK(k;\bbZ/l^\nu) \to \bbK(\cC^{(d)}_A; \bbZ/l^\nu) \to \bigoplus^v_{r=1} \Sigma \bbK(k;\bbZ/l^\nu)\,,$$
where $v$ stands for the number of simple (right) $A$-modules and $\Phi_A$ for the inverse of the Coxeter matrix of $A$. Consider the (matrix) homomorphism
\begin{equation}\label{eq:homomorphism-1}
(-1)^d \Phi_A - \Id \colon \bigoplus^v_{r=1} \bbZ/l^\nu \too \bigoplus^v_{r=1} \bbZ/l^\nu\,.
\end{equation}
As proved by Suslin in \cite[Cor.~3.13]{Suslin}, we have $\bbK_i(k;\bbZ/l^\nu)\simeq \bbZ/l^\nu$ when $i\geq 0$ is even and $\bbK_i(k;\bbZ/l^\nu)=0$ otherwise. Consequently, making use of the long exact sequence of algebraic $K$-theory groups with coefficients associated to the above distinguished triangle of spectra, we obtain the following computations:
$$
\bbK_i(\cC_A^{(d)}; \bbZ/l^\nu) \simeq \left\{ \begin{array}{lll}
\mathrm{cokernel}\,\,\mathrm{of}\,\, \eqref{eq:homomorphism-1}& \mathrm{if} & i\geq 0\,\,\mathrm{even} \\
\mathrm{kernel}\,\,\mathrm{of}\,\,\eqref{eq:homomorphism-1} & \mathrm{if} & i\geq 0 \,\,\mathrm{odd} \\
0 & \mathrm{if} & i <0\,.
\end{array} \right.
$$
Consider also the following dg functors
\begin{eqnarray*}
S^{-1} \Sigma^d\colon \cD^b_\dg(A) \too \cD^b_\dg(A) && d \geq 0\,,
\end{eqnarray*}
where $S$ stands for the Serre dg functor. The associated dg orbit category $\cC_d(A):=\cD^b_\dg(A)/(S^{-1} \Sigma^d)^\bbZ$ is usually called the {\em generalized $d$-cluster dg category of $A$}; see \cite[\S1.3]{AIR} and the references therein. Since $S^{-1}\Sigma=\tau^{-1}$, we have $\cC_d(A)\simeq \cC_A^{(d-1)}$.

Now, let us take for $A$ the $k$-algebra $kQ/\langle \rho\rangle$ associated to the quiver with relations $(Q,\rho)$. As proved independently by Amiot-Iyama-Reiten \cite[\S5]{AIR} and Thanhoffer de V\"olcsey-Van den Bergh \cite{VB}, we have $\cD^{\mathrm{sing}}_\dg(X)\simeq \cC_{d-1}(A)$. Consequently, it remains then only to show that the homomorphism \eqref{eq:homomorphism-1}, with $d$ replaced by $d-2$, agrees with the homomorphism $\mathrm{M}$ associated to the matrix $M:=(-1)^{d-1}C(C^{-1})^T - \Id$. On the one hand, the number of simple (right) $A$-modules agrees with the number of vertices of the quiver $Q$. This implies that $v=n-1$. On the other hand, the inverse of the Coxeter matrix of $A$ can be expressed as $-C(C^{-1})^T$, where $C_{ij}$ equals the number of arrows in $Q$ from $j$ to $i$ (counted modulo the relations). This implies that $(-1)^{d-2}\Phi_A-\Id= (-1)^{d-1}C(C^{-1})^T -\Id = \mathrm{M}$, and hence concludes the proof. 

\medbreak\noindent\textbf{Acknowledgments:}
The author is grateful to Michel Van den Bergh for useful discussions.

\end{document}

\end{proof}